\documentclass{nm}
\renewcommand\thisnumber{x}
\renewcommand\thisyear {201x}

\renewcommand\thisvolume{xx}
\usepackage{charter}
\usepackage[charter]{mathdesign}
\usepackage{amsmath}
\usepackage{graphicx}%
\usepackage{latexsym ,rawfonts}
 \usepackage{multirow}

 \usepackage{bm} 
 \usepackage{algorithmic,algorithm}

 \newcommand{\zipo}{\ensuremath{\text{BA}_{1/x}}}
 \newcommand{\ds}{\displaystyle}
\usepackage{listings}
\usepackage{framed}
\usepackage{xcolor}
\colorlet{shadecolor}{gray!20}
\lstnewenvironment{code}
{\lstset{language=mathematica,mathescape,columns=flexible}}{}
 \usepackage{color}
 \usepackage{ulem}
 \definecolor{cola}{rgb}{.0,.4,.9}

 \newcommand{\bmat}{\begin{pmatrix}}
 \newcommand{\emat}{\end{pmatrix}}

\newcommand{\lambdastar}{\ensuremath{\lambda_0^*}}
 \newcommand{\Reals}[1]{{\rm I\! R}^{#1}}

 \newcommand{\cq}{\ensuremath{p}}

\newtheorem{result}[theorem]{Result}

 \numberwithin{equation}{section}
 \numberwithin{table}{section}
 \numberwithin{figure}{section}
 \everymath{\displaystyle}

\begin{document}


 \markboth{J.Brannick, X. Hu, C. Rodrigo, and L. Zikatanov}{LFA for parallel smoothers and aggressive coarsening}
 \title{Local Fourier analysis of multigrid
 methods with polynomial smoothers and aggressive coarsening}


 \author[James Brannick, Xiaozhe Hu, Carmen Rodrigo, and Ludmil
 Zikatanov]{James Brannick\affil{1}, Xiaozhe
   Hu\affil{1}\comma\corrauth,  Carmen Rodrigo\affil{2}, and
   Ludmil Zikatanov\affil{1}\comma\affil{3}}
 \address{\affilnum{1}\ Department of Mathematics, The
       Pennsylvania State University, University Park, PA 16802, USA.
 \\
 \affilnum{2}\ Department of Applied Mathematics, University of Zaragoza,
 C/ Maria de Luna 3, 50018, Zaragoza, Spain.\\
 \affilnum{3}\ Institute of Mathematics and Informatics, Bulgarian
 Academy of Sciences, Acad.~G.~Bonchev Str., Bl. 8, 1113 Sofia, Bulgaria.
}

 \emails{
 {\tt brannick@psu.edu} (James Brannick),
 {\tt hu\_x@math.psu.edu} (Xiaozhe Hu), {\tt carmenr@unizar.es} (Carmen
 Rodrigo),
 {\tt ludmil@psu.edu} (Ludmil Zikatanov)
 }
 %

 \begin{abstract}
   We focus on the study of multigrid methods with aggressive
   coarsening and polynomial smoothers for the solution of the linear
   systems corresponding to finite difference/element discretizations
   of the Laplace equation.  Using local Fourier analysis we determine
   \textit{automatically} the optimal values for the parameters
   involved in defining the polynomial smoothers and achieve
   fast convergence of cycles with aggressive coarsening. We also
   present numerical tests supporting the theoretical results and the
   heuristic ideas. The methods we introduce are highly parallelizable
   and efficient multigrid algorithms on structured and
   semi-structured grids in two and three spatial dimensions.
 \end{abstract}



 \maketitle

 \section{Introduction}

 For emerging many-core parallel architectures it has been observed
 that visiting the coarser levels of a multilevel hierarchy leads to a
 loss in performance, as measured by the percentage of peak performance achieved by
 the multigrid solver on such architectures.  Roughly speaking, on the
 finer levels, computing residuals and smoothing can achieve relatively
 high performance, whereas on the coarser levels the performance of
 multigrid degrades due to the fact that fewer of the active threads
 are needed for computation there.  These observations motivate the
 further study and development of multigrid methods that apply more
 smoothing on the finer levels together with aggressive coarsening
 strategies.

 The use of point-wise smoothers (e.g., Jacobi and Gauss Seidel)
 together with aggressive coarsening in a multigrid solver has been
 studied using local Fourier analysis (LFA)
 in~\cite{2007Bin-ZubairH_OosterleeC_WienandsR-aa,2005WienandsR_JoppichW-aa}
 for rectangular grids
 and~\cite{2009GasparF_GraciaJ_LisbonaF_RodrigoC-aa} for triangular
 grids.  In these works, it has been observed that using aggressive
 coarsening is less efficient in terms of the total number of floating
 point operations than a more gradual coarsening approach, since these
 standard smoothing iterations are not able to effectively reduce a
 sufficiently large subspace of the high frequency components of the
 error.

 However, polynomial smoothers are well suited for aggressive coarsening
 approach since they can be constructed to achieve a preset
 convergence rate on a given subspace corresponding, for example, to
 a subinterval of the high frequency components of the error.  As
 shown in~\cite{2013_Vanek_Brezina, 2012BrezinaM_VassilevskiP-aa,
     2012KrausJ_VassilevskiP_ZikatanovL-aa, Baker.A;Falgout.R;Kolev.T;Yang.U.2011a}
by using a sufficiently large degree in the polynomial approximation it is
 possible to guarantee prescribed damping on a preset subinterval of
 high frequency components. These works contain important results
  and provide efficient algorithms by adjusting the polynomial degree
  for a given coarsening ratio.

  The focus of our work is on determining precisely the parameters of
  the polynomial smoothers, such as intervals of approximation,
  damping factors for high frequencies, coarsening ratios which result
  in best possible convergence rate. This, of course is an ambitious
  goal, but for semi-structured triangular and also rectangular
  grids this can be done. Our idea is to use the local Fourier
  analysis (LFA) to automatically determine the smoother and
  coarsening parameters which result in best performance.  As we show,
  LFA allows us to obtain quantitative estimates of the
  performance of multigrid methods with polynomial smoothers of
  arbitrary degree and aggressive coarsening.  As shown in
  \cite{Baker.A;Falgout.R;Kolev.T;Yang.U.2011a,Baker.A;Falgout.R;Kolev.T;Yang.U.2011b}
  polynomial smoothers result in   algorithms with high degree of
  parallelism  and outperform algorithms based on more classical
  relaxation methods. This is an additional advantage of the
  algorithms studied here as well.

  The paper is organized as follows. We review some basic facts about
  two-grid and multigrid iterations in Section~\ref{sect:tg-mg}. Next,
  in Section~\ref{sect:poly-smoothers} we introduce the polynomial
  smoothers of interest -- all based on Chebyshev polynomials arising
  as solutions to different minimization problems -- (1) appropriately
  shifted and scaled classical Chebyshev polynomials, (2) the so
  called smoothed aggregation polynomial, used as a smoother
  in~\cite{2012BrezinaM_VassilevskiP-aa}, or (3) the best polynomial
  approximation to $x^{-1}$ which is proposed as a smoother
  in~\cite{2012KrausJ_VassilevskiP_ZikatanovL-aa}.
 The local Fourier analysis for these polynomial smoothers, together
 with a two-grid LFA for aggressive coarsening are
 presented in Section~\ref{sect:LFA}, as are their extensions to
 triangular grids. Then bounds on the smoothing factors for the
 polynomials are calculated in Section~\ref{sect:bounds_LFA}. In this
 section we also show how to utilize the LFA results in choosing the
 optimal parameters for the corresponding polynomial smoother. In
 Section~\ref{sect:numerics} we present numerical tests illustrating
 the findings in the previous sections and also we provide extension
 to triangular grids in section~\ref{sect:triangular_numerics}. Finally
 we draw some conclusions in Section~\ref{sect:conclusions}.

 \section{Two-grid and multigrid iterations\label{sect:tg-mg}}
 A variational two-grid (two-level) method with one post smoothing step is defined as follows.
Given an approximation $w\in V$ to the solution $u$
 of the system $Au=f$, an update $v\in V$
 is computed in two steps
 \begin{enumerate}
 \item [(1)] $y = w + PA_H^{-1}P^T(f-Aw)$.
 \item [(2)] $v = y+R(f-Ay)$.
 \end{enumerate}
 Step~(1) is the coarse-grid correction iteration and step~(2) is
 called smoothing step and the operator $R$ is chosen so that it is
 convergent in $A$-norm, namely $\parallel I -RA \parallel_A < 1$.
 The corresponding error propagation operator of the two-level method
 is given by
 $$
 E_{TL} = (I-RA)(I-\pi_A), \quad \pi_A = PA_H^{-1} P^T A,
 $$
 where $A_H$ denotes the coarse-level operator, obtained by
 re-discretizing the problem on the coarse mesh, or, more generally,
 via Galerkin definition $A_H = P^TAP$ where $P$ is the prolongation
 matrix and $R$ denotes the smoothing operator.  Note that if
 $\bar{R}= R+R^T - R^T A R$, then $\left( I - \bar{R} A \right) = ( I
 - R^T A )( I - R A )$, and if $\bar{R}$ is SPD, then the smoother is
 convergent in $A$-norm. We will consider smoothers for which $R=R^T$
 and in this case $\bar{R} = 2R-RAR$.

The above two-level method can be easily generalized to multilevel methods by simple recursion.  Suppose we already have defined the multilevel method on the coarse level, which is denoted by $B_H$ (on the coarsest level, we solve the linear system exactly, i.e. $B_H$ on the coarsest level is defined by $A_H^{-1}$), then the multilevel method on the fine level can be defined as follows:
 \begin{enumerate}
 \item [(1)] $y = w + PB_HP^T(f-Aw)$.
 \item [(2)] $v = y+R(f-Ay)$.
 \end{enumerate}
 It is easy to see the corresponding error propagation operator for the multilevel method is given by
 $$
 E_{ML} = I - BA = (I-RA)(I-PB_HP^TA).
 $$
Note that, the multilevel method, denoted by $B$, is defined recursively through the multilevel method $B_H$ defined on the coarse grid.


 \section{Simple preconditioners and
smoothers based on Chebyshev polynomials\label{sect:poly-smoothers}}

To set the terminology, we call any SPD operator $R_0$ that
approximates $A^{-1}$ a preconditioner.  
A simple example of a preconditioner $R_0$ is the inverse of the
diagonal of $A$, i.e. $R_0=D^{-1}$. Another example is furnished by
the $\ell_1$-Jacobi preconditioner introduced
in~\cite{2009KolevT_VassilevskiP-aa}:
 $$R_0 = R_{\ell_1} = \operatorname{diag}(r_{11},r_{22},\hdots ), \quad
 r_{ii} = 1/(a_{ii} + \sum_{j\neq i} |a_{ij}|).$$
 Note that $R_0$ in
 such case is also a convergent smoother in the terminology given above.
 A weighted version of $R_{\ell_1}$ can be found in~\cite{2012BrezinaM_VassilevskiP-aa}.

 Next, for any given preconditioner we can construct a convergent
 smoother in the following way: Given a preconditioner $R_0$,
 let $\lambda_1$ be a given bound on the spectral radius of $R_0A$,
 that is, $\rho(R_0A)\le \lambda_1$ and
 $q_\nu(x)\in \mathcal{P}_\nu$ be a polynomial of degree $\nu$ such that
 \begin{equation}\label{poly-bounds}
 | 1-xq_\nu(x) | < 1, \quad \mbox{for all}\quad x\in [0,\lambda_1].
 \end{equation}
 We then set
 \begin{equation}\label{preconditioner-to-smoother}
 R=q_\nu (R_0A)R_0\quad\mbox{or equivalently}\quad
 R = R_0^{\frac12}q_\nu(R_0^{\frac12}AR_0^{\frac12})R_0^{\frac12}.
 \end{equation}
 Note that  this is a symmetric,  and, by~\eqref{poly-bounds}, a
 convergent smoother.

 Various polynomial approximations and the resulting parallel
 smoothers have been designed (also in the context of aggressive
 coarsening), see
 e.g.~\cite{Adams03parallelmultigrid, Brezina99,2012BrezinaM_VanekP_VassilevskiP-aa,2012KrausJ_VassilevskiP_ZikatanovL-aa,stuebenmultigrid, Vanek1996}. 

 We consider three different polynomial smoothers given
 in~\eqref{preconditioner-to-smoother}. As  all these are based on Chebyshev
 polynomials of first kind we recall the definitions of the $k$-th
 such polynomial $T_k(t)$ and the following recurrence relation
 $$
 T_0(t) = 1, \quad T_1(t) = t, \quad T_k(t) = 2tT_{k-1}(t) - T_{k-2}(t).
 $$
Equivalently we can compute that
 \[
 T_k(t) =
 \frac12\left(t-\sqrt{t^2-1}\right)^k +
 \frac12\left(t+\sqrt{t^2-1}\right)^k.
 \]
 Clearly, $|T_k(t)|\le 1$, $t\in [-1,1]$, and $T_k(t)=\pm 1$ at $(k+1)$
 different points in $[-1,1]$. It is clear then that $|T_k(t)|$ is
 strictly monotone for $t\notin (-1,1)$.

 We now define the three sequences of polynomials $\{q_\nu\}$ that we
 consider in this study. We fix $R_0$ and $\lambda_1\ge\rho(R_0A)$.
 Generally, we aim to construct a smoother $q_\nu(x)$ such that the
 polynomial $p_{\nu+1} (x) = 1 - x q_{\nu}(x)$ is small on the
 interval $[\lambda_0,\lambda_1]$, with $0\le \lambda_0$. Here $\nu$
 denotes the degree of the polynomial smoother, $\lambda_1$ is an
 approximation to the largest eigenvalue of $R_0A$, satisfying
 $\lambda_1>\rho(R_0A)$ and $\lambda_0$ is a parameter that can be
 adjusted to define the smoothing interval.  In some of the cases we
 set $\lambda_0=\frac{\lambda_1}{\kappa}$, with $\kappa>1$, and in
 other cases we take $\lambda_0=0$.  The smoothing factors clearly
 depend on the choice of $\lambda_0$. How to choose this parameter
 optimally is not obvious and this is discussed and addressed in our study.

The scaled and shifted classical Chebyshev polynomial solves
the following minimization problem:
 \begin{equation}\label{cheb-1-min}
 p_{\nu+1}(x) = \arg\min\{ \|p\|_{\infty,[\lambda_0,\lambda_1]},
 \quad p\in \mathcal{P}_{\nu+1}, \quad\mbox{such  that}\quad
 p(0)=1\}.
 \end{equation}
and we have then for $q_\nu(x)$:
 \begin{equation}\label{cheb-1}
 q_\nu(x) = \frac{1-p_{\nu+1}(x)}{x}, \quad
 p_{\nu+1}(x) =
 \frac{T_{\nu+1}\left(\frac{\lambda_0+\lambda_1-2x}{\lambda_1-\lambda_0}\right)}
 {T_{\nu+1}\left(\frac{\lambda_0+\lambda_1}{\lambda_1-\lambda_0}\right)}.
 \end{equation}
 Note that $p_{\nu+1}(x) = 1-x q_\nu(x)$, and from the monotonicity of
 $T_{\nu+1}(t)$ for $t\notin(-1,1)$ we have $|p_{\nu+1}(x)|<1$, for
 all $x\in [0,\lambda_1]$. This inequality holds independently of the
 choice of $\lambda_0\in [0,\lambda_1)$. We also refer
 to~\cite{amli-1,panayot-book} for more details on these smoothers.

 Using the three term recurrence relation of Chebyshev polynomials, it is straightforward to get the following identity
 \begin{equation*}
 q_\nu(x) - q_{\nu-1}(x) = \frac{2\zeta aT_\nu(a)}{T_{\nu+1}(a)}(1-xq_{\nu-1}(x)) + \frac{T_{\nu-1}(a)}{T_{\nu+1}(a)}(q_{\nu-1}(x) - q_{\nu-2}(x)).
 \end{equation*}
 where $\zeta = \frac{2}{\lambda_1+\lambda_0}$ and $a = \frac{\lambda_1+\lambda_0}{\lambda_1 - \lambda_0}$.  This identity leads to the following algorithm which computes
 \begin{equation*}
 Rr = q_{\mu}(D^{-1}A)D^{-1}r.
 \end{equation*}

 \begin{algorithm}
 \caption{Chebyshev Polynomial Preconditioning with $R = q_{\mu}(D^{-1}A)D^{-1}$}
 \begin{algorithmic}[1]
 \STATE $\bar{r} \gets D^{-1}r$,
 \STATE $v_0 \gets a\bar{r}$,
 \STATE $v_1 \gets \frac{4\zeta a^2}{2a^2-1} \bar{r} - \frac{2\zeta^2 a^2}{2a^2-1}D^{-1}A \bar{r}$,
 \FOR{$j=2, 3, \ldots,m$}
 \STATE $\bar{r}_{j-1} \gets D^{-1}(r - Av_{j-1})$,
 \STATE $v_j = v_{j-1} + \frac{T_{j-1}(a)}{T_{j+1}(a)} (v_{j-1} - v_{j-2}) + \frac{2 \zeta a T_{j}(a)}{T_{j+1}(a)} \bar{r}_{j-1}$.
 \ENDFOR
 \end{algorithmic}
 \end{algorithm}

 To define the smoothed aggregation polynomial smoother
 (see~\cite{2012BrezinaM_VassilevskiP-aa}), we use a minimization
 problem to define first $p_{\nu+1}(x)$:
 \begin{equation}\label{sa-1-min}
 p_{\nu+1} = \arg\min\{ \|p\sqrt{x}\|_{\infty,[0,\lambda_1]},
 \quad p\in \mathcal{P}_{\nu+1}, \quad\mbox{such  that}\quad
 p(0)=1\}.
 \end{equation}
 The smoothed aggregation polynomial $q_\nu$ is then defined as follows:
 \begin{equation}\label{sa-1}
 q_\nu(x) = \frac{1-p_{\nu+1}(x)}{x}, \quad
 p_{\nu+1}(x) =
 (-1)^\nu \bigg( \frac{1}{2\nu+3}\bigg)\bigg(\frac{\sqrt{\lambda_1}}{\sqrt{x}}\bigg)T_{2\nu+3}
 \bigg(\frac{\sqrt{x}}{\sqrt{\lambda_1}}\bigg).
 \end{equation}
 Note that this formulation does not require the parameter
 $\lambda_0$.  We refer to~\cite{Vanek1996,
   2012BrezinaM_VassilevskiP-aa,2012BrezinaM_VanekP_VassilevskiP-aa}
 for additional properties of this smoother.


Finally, we mention the smoother based on the best polynomial
approximation to $1/x$ in uniform norm
from~\cite{2012KrausJ_VassilevskiP_ZikatanovL-aa}.  The minimization
problem now directly defines $q_\nu(x)$
and
\begin{equation}\label{1x-1-min}
q_\nu(x)= \arg\min\{ \|x^{-1}-p\|_{\infty,[\lambda_0,\lambda_1]}, \quad p\in \mathcal{P}_{\nu}\}.
\end{equation}
Introducing $p_{\nu+1}(x)$ as in the previous cases, we note that this
formulation thus biases the approximation to small values of $x$.  The
polynomial $q_\nu(x)$ is computed using three-term recurrence
relation. For details on this polynomial and its implementation we refer
to~\cite{2012KrausJ_VassilevskiP_ZikatanovL-aa}.

\subsection{Local Fourier analysis for polynomial smoothers}\label{sect:LFA}
We now describe briefly the technique known as Local Fourier Analysis
(LFA), introduced by Brandt in \cite{1977BrandtA-aa}. This technique
is considered to be a useful tool in providing quantitative
convergence estimates for idealized multigrid algorithms. Such
estimates can be rigorously justified in cases when the boundary
conditions are periodic. It is also known that for structured or
semi-structured grids the LFA provides
accurate predictions for the asymptotic convergence rates of multigrid
methods for problems with other types of boundary
conditions as well.  The analysis is based on the Discrete Fourier transform,
and a good introduction to such analysis is found in the monographs by
Trottenberg et al.~\cite{TOS01}, and Wienands and
Joppich~\cite{2005WienandsR_JoppichW-aa}.

The main idea of the LFA is to formally extend all multigrid
components to an infinite grid, neglecting the boundary conditions,
and analyze discrete linear operators with constant coefficients. In
this way, the eigenfunctions of such operators are the eigenfunctions
of the shift operators, namely, $S_{\pm h}\varphi = \varphi(\cdot\pm
h)$, called Fourier components. If we assume that the error is a
linear combination of the Fourier components, then the behavior of a
multigrid algorithm can be studied by looking at the reduction on each
one of these components.  Although this analysis seems to be somewhat
heuristic, its practical value has been widely recognized. In general,
the LFA does not only provide accurate asymptotic convergence rates,
but also provides the means to select optimal components for the
multigrid algorithm.

In this section, we present a suitable local Fourier analysis
technique to derive quantitative estimates for the convergence of
multigrid methods with polynomial smoothers and aggressive
coarsening. In particular, a smoothing analysis for polynomial
smoothers and a two-level analysis by considering aggressive
coarsening from a grid with step-size $h$ to a coarse-grid of size
$2^k h$, are introduced next.  We begin by setting up the LFA
framework. We extend the discrete problem $A_h u_h = f_h$ to an
infinite grid
\begin{equation}\label{infinite_grid}
\Omega_h = \{{\mathbf x}=(x_1, x_2)\, \vert \, x_i=k_i h_i, \;k_i\in {\mathbb Z}, \ i=1,2 \},
\end{equation}
where ${\bf h} = (h_1, h_2)$ is the grid spacing.  From the definition
of the operators on $\Omega_h$, the discrete solution, its current
approximation and the corresponding error or residual can be
represented by formal linear combinations of the Fourier modes:
$\varphi_h({\bm \theta}, {\mathbf x}) ={\rm e}^{i \theta_1
  x_1} \, {\rm e}^{i \theta_2 x_2}$, with ${\mathbf x}\in \Omega_h$,
and ${\bm\theta}=(\theta_1,\theta_2)\in {\bm
  \Theta}_h= (-\pi/h_1,\pi/h_1]\times(-\pi/h_2,\pi/h_2]$. These grid
functions form a unitary basis for the space of bounded
functions on the infinite grid, and define the Fourier space
$$
{\mathcal F}(\Omega_h) := span \{\varphi_h({\bm \theta,\cdot})
\, \vert \, {\bm \theta}\in {\bm \Theta}_h\}.
$$
In this way, the behavior of the multigrid method can be analyzed by
evaluating the error reduction associated with a particular multigrid
component on the Fourier modes. Clearly, the discrete operator, and,
also the discrete error transfer operator have simpler form when
expressed in the Fourier basis. For instance, the symbol of $A_h$,
denoted by $\widetilde{A}_h$, is typically (block) diagonal, and each
diagonal element corresponds to a particular ``frequency''. In what
follows, we denote by $\widetilde{X}$ the Fourier symbol
of a given operator $X$.

To perform a smoothing or a two-grid analysis, we distinguish high-
and low-frequency components on $\Omega_h$. The classification of
``high'' and ``low'' here is done with respect to  the coarse
grid, since some Fourier components are not ``visible'' on the coarse
grid. Usually such ``invisible'' modes
are zero at the coarse grid degrees of freedom, or,
they are orthogonal (in an appropriate scalar product) to all Fourier modes
corresponding to the coarser grid.

Here, we consider aggressive coarsening techniques and the
coarse grid is denoted by $\Omega_{2^k h}$, with $k$ characterizing
how ``aggressive'' the coarsening is.  Then, the range of
frequencies with respect to this coarse grid is defined as
\begin{equation}\label{eq:frequencies}
\begin{array}{rcl}
\mbox{all frequencies:}\quad  &&
{\bm \Theta}_h= (-\pi/h_1,\pi/h_1]\times(-\pi/h_2,\pi/h_2],\\
\mbox{low frequencies:}\quad  &&{\bm  \Theta}_{2^k h} = (-\pi/(2^k
h_1),\pi/(2^k h_1)]\times (-\pi/(2^k h_2),\pi/(2^k h_2)],\\
\mbox{high frequencies:}\quad  && {\bm  \Theta}_h\backslash{\bm \Theta}_{2^k h}.
\end{array}
\end{equation}

In order to investigate the action of the smoothing operator on the
high-frequency error components we use a technique known as
smoothing analysis. 

We consider a splitting of the discrete operator $A_h =
A_h^+ + A_h^-$, where the splitting defines the smoothing iteration:
with a given initial guess $u_{h,0}$ we define $u_{h,k+1}$ in terms of
$u_{h,k}$ as follows:
\[
\bm{u}_{h,k+1} = \bm{u}_{h,k} + [A_h^{+}]^{-1}(f-A_h u_{h,k}).
\]
Smoothing analysis can be performed for many choices (symmetric or
non-symmetric) of $[A_h^+]^{-1}$, but to tie this to our earlier
discussion, the smoothers we are interested in are given by
$R_0=[A_h^+]^{-1}$ in \eqref{preconditioner-to-smoother} with SPD
$A_h^{+}$. For example, such smoothers are the Jacobi method, i.e.
$A_h^{+}=D=\text{diag}(A_h)$ and the $\ell_1$-smoother mentioned
earlier.

For the polynomial smoothers, which are of interest here, the error
propagation operator is
\begin{equation}\label{S_h}
S_h=I_h-\cq_m(X)X, \quad \mbox{with}\quad X=[A_h^+]^{-1}A_h
\end{equation}
where $\cq_m$ is a polynomial of degree $m$, positive on the spectrum of
$A_h$.
Since Fourier modes are eigenfunctions of the smoothing operator, we
can estimate the smoothing factor of $S_h$, i.e. the error reduction
in the space of high-frequencies $\widetilde{S}_h({\bm \theta})$, as follows
\begin{equation}\label{smoothing_factor}
\mu = \sup_{{\bm \Theta}_h\backslash{\bm \Theta}_{2^k h}} |\widetilde{S}_h^{\nu}({\bm \theta})|,
\end{equation}
where $\nu$ denotes the number of iterations of the relaxation
process.

\subsection{Two-grid LFA for aggressive coarsening}\label{sect:LFA_TL_aggresive}

The basis for the efficient performance of a multigrid method is the
interplay between the smoothing and the coarse-grid correction parts
of the algorithm. Thus, to get more insight in the behavior of a
multigrid algorithm, it is convenient to perform at least a two-grid
analysis which takes into account the influence of the components
involved in the coarse-grid correction. This is even more important when
aggressive coarsening strategies are applied, since the smoother
should be chosen according to the factor we are coarsening with. With
this purpose, in this section we present a two-grid local Fourier
analysis which considers an aggressive coarsening strategy from a grid
with step-size $h$ to a grid with step-size $2^k h$.

To perform this analysis, we consider the fine- and coarse-grids $\Omega_h$ and $\Omega_{2^k h}$, respectively.
It is well known that in the transition from the fine- to the coarse-grid, each low-frequency is coupled with several high-frequencies. In particular, for an arbitrary $k$, each low-frequency ${\bm \theta}\in \bm{\Theta}_{2^kh} = (-\pi/(2^kh),\pi/(2^kh)]\times (-\pi/(2^kh),\pi/(2^kh)]$ is coupled with $4^k-1$ high-frequencies by the coarse-grid correction operator. Because of this, the Fourier space can be subdivided into the corresponding $4^{k}-$dimensional subspaces ${\mathcal F}^{2^kh}({\bm \theta})$ which are generated by the Fourier modes associated with these $4^{k}$ frequencies, in the way that $\displaystyle{\mathcal F}(\Omega_h) = \bigoplus_{{\bm \theta}\in\bm{\Theta}_{2^kh}} {\mathcal F}^{2^kh}({\bm \theta})$.
Therefore, a block-matrix representation of the two-grid operator on the Fourier space can be obtained, which simplifies the computation of the spectral radius of the iteration matrix of the method, since the result should be the maximum of the spectral radius of the corresponding blocks.

Let $B_h^{2^k h}$ be the error propagation matrix of the considered two-grid method, that is, $e_h^{m+1} = B_h^{2^k h} e_h^{m}$, given by
\begin{equation}\label{two_grid_it}
B_h^{2^k h} = S_h^{\nu_2}C_h^{2^k h}S_h^{\nu_1} = S_h^{\nu_2}(I_h-P^h_{2^k h}(A_{2^k h})^{-1}R_h^{2^k h}A_h)S_h^{\nu_1},
\end{equation}
where $S_h$ is the iteration matrix associated with the smoother,
$\nu_1$\,, $\nu_2$ are, respectively, the number of pre- and
post-smoothing steps, and $C_h^{2^k h}$ is the coarse-grid correction
operator, mainly composed of the inter-grid transfer operators:
$P^h_{2^k h}$, $R_h^{2^k h}$ restriction and prolongation,
respectively, and the discrete operators $A_h$, $A_{2^k h}$ on the
fine and coarse grids, respectively.  For standard relaxation schemes
and, in particular, for the polynomial smoothers used here the
two-grid operator $B_h^{2^k h}$ leaves invariant the subspaces
${\mathcal F}^{2^kh}({\bm \theta})$. As a consequence this operator
can be represented by a block-diagonal matrix, consisting of
$(4^{k}\times4^{k})-$blocks, denoted by
$$\widetilde{B}_h^{2^k h}({\bm \theta}) = (\widetilde{S}_h({\bm
  \theta}))^{\nu_2}\widetilde{C}_h^{2^k h}({\bm
  \theta})(\widetilde{S}_h({\bm \theta}))^{\nu_1}.
$$
Here, ${\bm \theta}\in\bm{\Theta}_{2^kh}$, $\widetilde{S}_h({\bm
  \theta})$, and $\widetilde{C}_h^{2^k h}({\bm \theta})$ are the
block-matrix representations in the subspaces ${\mathcal
  F}^{2^kh}({\bm \theta})$ of the smoother and the coarse-grid
correction operator. The latter is computed from the Fourier
representations of the coarse-grid correction, namely,
\begin{equation}\label{cgc_symbol}
  \widetilde{C}_h^{2^k h} = \widetilde{I}_h-\widetilde{P}^h_{2^k h}(\widetilde{A}_{2^k h})^{-1}\widetilde{R}_h^{2^k h}\widetilde{A}_h.
\end{equation}
Now, the local Fourier analysis prediction for the asymptotic two-grid convergence factor of the method can be determined as:
\begin{equation}\label{def_rho}
\rho=\rho(B_h^{2^k h}) = \max_{{\bm \theta}\in\bm{\Theta}_{2^kh}}\rho
(\widetilde{B}_h^{2^k h}({\bm  \theta})).
\end{equation}

\subsection{Extension to triangular grids}\label{sect:triangular_grids}

To analyze the influence of grid-geometry on the behavior of the
multigrid method, in this section we extend the ideas to triangular
grids. We consider the discretization of the Laplace operator by
linear finite elements on a regular triangulation of a general
triangle.
\begin{figure}[htb]
\begin{center}
\includegraphics[width=9cm]{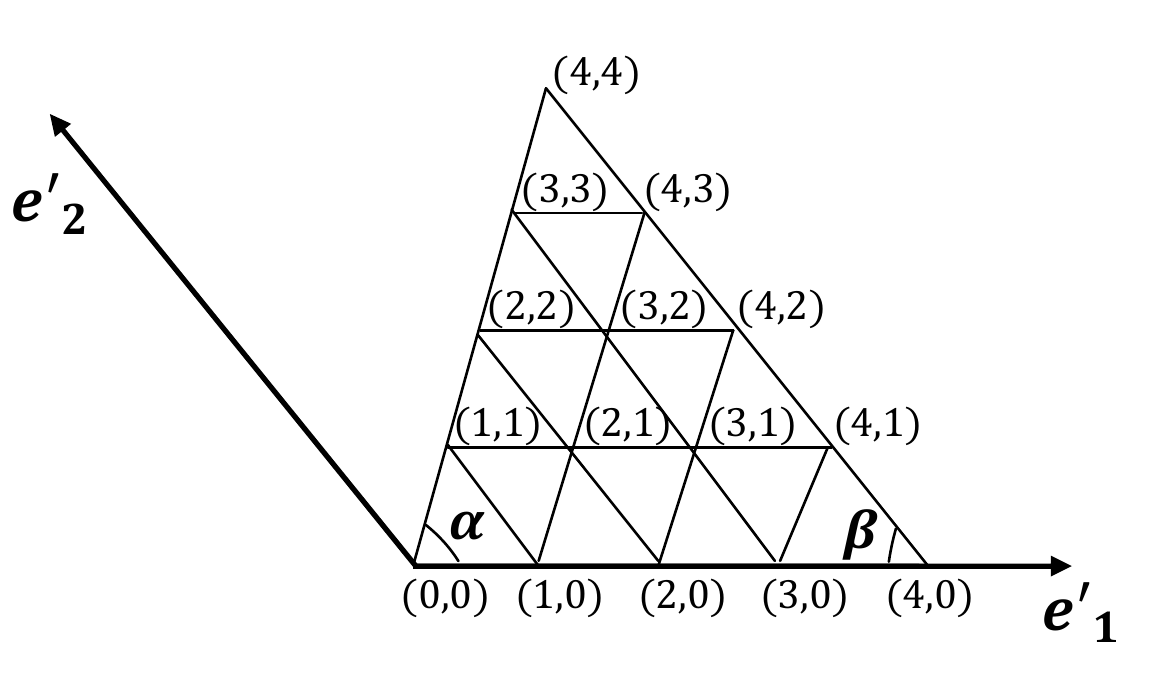}
\end{center}
\caption{Basis in ${\mathbb R}^2$ fitting the geometry of a triangular grid characterized by two angles, and local enumeration for the regular grid obtained after two refinement levels.}
\label{triangular_grid_angles}
\end{figure}
This triangular grid is characterized by two angles $\alpha$ and
$\beta$ and a local enumeration with double index is fixed by
considering a unitary basis of ${\mathbb R}^2$, $\{{\mathbf
  e}_1',{\mathbf e}_2'\}$, fitting the geometry of the triangular grid
(see Figure~\ref{triangular_grid_angles}).  We fix the
axis-orientation $\{{\mathbf e}_1',{\mathbf e}_2'\}$, the
corresponding enumeration of vertices as in
Figure~\ref{triangular_grid_angles}, and we set $\gamma =
\pi-\alpha-\beta$.  The stencil form of the discrete operator then reads
(see, e.g.~\cite[pp. 189--190]{2012RodrigoC_GasparF_LisbonaF-aa})
\begin{equation}\label{Laplace_triangular_grids}
  A_{h} = \ds\frac{\cot\alpha+\cot\beta}{h^2}
\begin{bmatrix}
    0 & -\cot\alpha & -\cot\beta\\
    -\cot\gamma & 2(\cot\alpha+\cot\beta+\cot\gamma) & -\cot\gamma\\
     -\cot\beta & -\cot\alpha & 0
\end{bmatrix}.
\end{equation}
Following \cite{Fourier_tri}, the local Fourier analysis applied to
discretizations on rectangular grids, can be extended to
discretizations on triangulations. The key to carrying out this
generalization is to define a new two-dimensional Fourier transform
using non-orthogonal bases. More precisely, we pick a spatial basis
that fits the structure of the grid (see
Figure~\ref{triangular_grid_angles}) and we chose its reciprocal basis
in the frequency space. In this way LFA on triangular grids is performed similarly
as on rectangular grids.


\section{Bounds via local Fourier analysis}\label{sect:bounds_LFA}
From the considerations in Section~\ref{sect:poly-smoothers} we know
that to construct smoothers based on scaled Chebyshev polynomials or
based on the \zipo polynomial we need to specify the interval
$[\lambda_0,\lambda_1]$. On this interval, as mentioned earlier the
polynomials have optimal properties and solve different minimization
problems.

In this section, we describe briefly how to optimize the choice of
$\lambda_0$ used in defining ${\mathcal Q}_m((A_h^+)^{-1}A_h)$.  One
simple choice is to use LFA for
$(\widetilde{A}_h^+(\bm{\theta}))^{-1}\widetilde{A}_h(\bm{\theta})$,
and define $\lambda_0$ and $\lambda_1$ as
\begin{equation}\label{lambda_0_1}
\lambda_{0} = \min_{\theta\in {\bm \Theta}_h\backslash{\bm \Theta}_{2^k h}} |(\widetilde{A}_h^+(\bm{\theta}))^{-1}\widetilde{A}_h(\bm{\theta})|, \qquad
\lambda_{1} = \max_{\theta\in {\bm \Theta}_h\backslash{\bm \Theta}_{2^k h}} |(\widetilde{A}_h^+(\bm{\theta}))^{-1}\widetilde{A}_h(\bm{\theta})|.
\end{equation}
This is a good choice, but is not optimal, as is evident from the
numerical results presented later. From the properties of the
classical Chebyshev polynomial smoothers and the \zipo smoother
introduced in \S\ref{sect:poly-smoothers} we see that two parameters,
$\lambda_0$ and $\lambda_1$, are used to define the smoother. Since we
would like to have scaling invariant smoothers we can fix one of these
parameters. For reasons which are evident from the analysis provided
in~\cite{amli-1, panayot-book} for the scaled Chebyshev smoother
and~\cite{2012KrausJ_VassilevskiP_ZikatanovL-aa} for the smoother
using the best approximation to $\frac1x$, the parameter that we fix
is $\lambda_1$. To optimize the performance of the smoother or the two
grid method, we try to adjust $\lambda_0$ in order to achieve one of
the following goals:
\begin{itemize}
\item get better smoothing properties, that is, minimize
  $|\widetilde{S}_h(\bm{\theta})|$ for $\bm{\theta}\in
  \bm{\Theta}_{h}\setminus\bm{\Theta}_{2^kh}$ (high frequency as
  defined in~\eqref{eq:frequencies}) to compensate for aggressive
  coarsening.
\item get the best possible two grid rate of convergence, namely, minimize
$|\widetilde{B}_h^{2^k h}(\bm{\theta})|$ for $\bm{\theta}\in \bm{\Theta}_h$.
\end{itemize}
We find an optimal value of $\lambda_0$ iteratively, choosing as
initial guess the values given in~\eqref{lambda_0_1}. This can be done
as explained in~the next section.  

\subsection{Optimal choice of $\lambda_0$}\label{sect:choice_lambda}
We now provide several of the relevant properties on this polynomial
smoother and discuss the choice of $\lambda_0$ for smoothers based on
the polynomial of best approximation of $x^{-1}$. For the estimates and
identities used below, we refer to
\cite{2012KrausJ_VassilevskiP_ZikatanovL-aa}. Let $\lambda_0>0$ be any
estimate of the minimal eigenvalue of $[A_h^+]^{-1} A_h$, for example
$\lambda_0=\lambda_1/\kappa$, for some $\kappa>1$.

We now discuss  the restrictions on the polynomial degree imposed
by the requirement that the smoother has certain error reduction and also
the requirement that the polynomial is positive on $(0,\lambda_1]$
(and as a consequence the matrix polynomial will be positive
definite). Further, let
\begin{eqnarray*}
E_m:= \max_{x\in [\lambda_0,\lambda_1]}|1-x\cq_m(x)|
=\max_{x\in [\lambda_0,\lambda_1]} x\cdot\left|\frac{1}{x}-\cq_m(x)\right|.
\end{eqnarray*}
As $\lambda_1$ is a point of Chebyshev alternance, \cite[Theorem~2.1
and Equation~(2.2)]{2012KrausJ_VassilevskiP_ZikatanovL-aa}, for the
error of approximation $E_m$ we have
\begin{eqnarray*}
E_m = \lambda_1
\left| \frac{1}{\lambda_1}- \cq_m(\lambda_1)\right|
= \left[\frac{2\lambda_1}{\lambda_1-\lambda_0}\right]\cdot \left[\frac{\delta^m}{a^2-1}\right] = \frac{2\kappa\delta^m}{(\kappa-1)(a^2-1)}.
\end{eqnarray*}
Here, we have denoted
\[
\kappa=\frac{\lambda_1}{\lambda_0},\quad
\quad \delta = \frac{\sqrt{\kappa}-1}{\sqrt{\kappa}+1},\quad
a = \frac{\kappa+1}{\kappa-1}.
\]
Computing the error $E_m$ then gives
\begin{equation}\label{eqn:max-lambda1}
E_m = \frac{\delta^m(\kappa-1)}{2}.
\end{equation}
Regarding the positivity, $\cq_m(\lambda_1)>0$, a sufficient
condition (and also necessary condition in many cases) is that
$\frac{1}{\lambda_1}-E_m >0$. Thus, we need to find the smallest $m$
such that both $E_m< \rho$, for a given $\rho$, and
$\cq_m(\lambda_1)>0$. We then have that the polynomial is
positive if
\begin{eqnarray*}
\frac{\delta^m(\kappa-1)}{2} \le  \frac{1}{\lambda_1}
\quad\Rightarrow\quad
\delta^m\le \frac{2}{\lambda_1(\kappa-1)}.
\end{eqnarray*}
We note that from this it follows that $R=\cq_m(A)$ and hence
$\bar{R}$ are symmetric and positive definite, implying that the
smoother is convergent in $A$-norm.

Also, a straightforward calculation shows that if we want a damping
factor less than $\rho$ on the interval $[\lambda_0,\lambda_1]$, we have
\begin{eqnarray*}
\frac{\delta^m(\kappa-1)}{2} \le  \rho
\quad\Rightarrow\quad
\delta^m\le \frac{2\rho}{\kappa-1}.
\end{eqnarray*}
Finally, the minimal $m$ that will have the desired properties satisfies
\begin{equation}\label{eq:bound-on-m}
m \ge
\frac{1}{|\log\delta|}\max\left\{ \left|\log \frac{2\rho}{\kappa-1}\right|,
\left|\log \frac{2}{\lambda_1(\kappa-1)}\right|
\right\}.
\end{equation}
Recall that $\lambda_1$ is fixed, and, hence,  both $\kappa$ and the polynomial degree $m$ are determined by
  $\lambda_0$. In short, once we choose $\lambda_0$ we can
  calculate the degree of the polynomial so that the resulting
  smoother has a  guaranteed convergence rate.
  This result is of interest in the context of aggressive coarsening
  since it allows to choose $\lambda_0$ in accordance with the
  coarsening ratio and the smoothing rate. Choosing $\lambda_0$ to
  optimize these two parameters naturally leads to computationally optimal methods.

Since the convergence rate of the polynomial smoother is determined by
$\max_{x \in [\lambda_0, \lambda_1]} | 1- x \cq_m(x) |$, the choice of
$\lambda_0$ using LFA (see~\eqref {lambda_0_1})
does not guarantee the optimal convergence rate of the polynomial
smoother.  Instead, we consider fixing $\lambda_1$ and the polynomial degree $m$, and finding the best lower bound, $\lambdastar$, that solves the following min-max problem:
\begin{equation} \label{eqn:min-optimal-lambda}
\min_{\lambda \in [\lambda_0, \lambda_1]} \max_{x \in [\lambda_0, \lambda_1]} | 1- x \cq_m(x; \lambda) | \longrightarrow \lambdastar ,
\end{equation}
where $\cq_m(x; \lambda)$ denote the best approximation of $1/x$ on
the interval $[\lambda, \lambda_1]$.  Next we will explain how to
solve this minimization problem and, therefore, obtain the optimal
$\lambdastar$.  First, the following lemma shows that the $| 1 - x
\cq_m(x; \lambda)| $, $\lambda \in [\lambda_0, \lambda_1] $ achieves
its maximum at the end points $\lambda_0$ or $\lambda_1$.

\begin{lemma}\label{lem:max-end}
If $\lambda \in [\lambda_0, \lambda_1]$, and $\cq_m(x; \lambda)$ is the best polynomial approximation to $1/x$ on $[\lambda, \lambda_1]$, then
\begin{equation}\label{eqn:max-end}
\max_{x \in [\lambda_0, \lambda_1]} |1 - x \cq_m(x, \lambda)| = \max \{  |1 - \lambda_1 \cq_m(\lambda_1; \lambda)|, | 1- \lambda_0 \cq_m(\lambda_0; \lambda) |   \}.
\end{equation}
\end{lemma}
\begin{proof}
  We first consider $x\in[\lambda, \lambda_1]$ and since $\cq_m(x;
  \lambda)$ is the best polynomial approximation to $1/x$ on
  $[\lambda, \lambda_1]$, we have that $\lambda_1$ is a point of Chebyshev
  alternance. Therefore,
\begin{equation*}
\max_{x \in [\lambda, \lambda_1]} | 1- x \cq_m(x; \lambda) | = | 1 - \lambda_1 \cq_m(\lambda_1; \lambda) |.
\end{equation*}
On the other hand, for $x \in [\lambda_0, \lambda]$, as shown in~\cite[Lemma~3.1 and
Equation~(3.5)]{2012KrausJ_VassilevskiP_ZikatanovL-aa}, $\cq_m(x; \lambda)$ is strictly decreasing on $(0, \lambda]$, and therefore,
\begin{equation*}
\max_{x\in[\lambda_0, \lambda]} | 1 - x \cq_m(x; \lambda) | =  | 1 - \lambda_0 \cq_m(\lambda_0; \lambda)  |.
\end{equation*}
Then~\eqref{eqn:max-end} follows directly and the proof is complete.
\end{proof}

Using Lemma~\ref{lem:max-end}, the minimization problem~\eqref{eqn:min-optimal-lambda}~can be simplified as following
\begin{equation}\label{eqn:opt-lambda}
\min_{\lambda \in [\lambda_0, \lambda_1]}  \max \{  |1 - \lambda_1 \cq_m(\lambda_1; \lambda)|, | 1- \lambda_0 \cq_m(\lambda_0; \lambda) |   \}.
\end{equation}

\subsection{Existence of optimal parameter}
In this subsection we provide convincing evidence that an optimal
$\lambdastar$ exists. We have not stated this as a theorem because we
have used Mathematica to verify the positivity of a derivative in the
last step of the proof below. We have the following result.

\begin{result}\label{thm:best-lower}
If $\lambda \in [\lambda_0, \lambda_1]$ and $\cq_m(x; \lambda)$ is the best polynomial approximation to $1/x$ on $[\lambda, \lambda_1]$, then there is a unique minimizer $\lambdastar$ of the minimization problem~\eqref{eqn:min-optimal-lambda}, and
\begin{equation}\label{eqn:opt-lambda}
|1 - \lambda_1 \cq_m(\lambda_1; \lambdastar )| = | 1- \lambda_0 \cq_m(\lambda_0; \lambdastar ) |.
\end{equation}
\end{result}
Here we outline an argument which at the last step uses a
bound verified numerically. As we have observed earlier, since $\lambda_1$ is a point of
Chebyshev alternance, we have
\begin{equation*}
| 1 - \lambda_1 \cq_m(\lambda_1; \lambda) | = \frac{  (\sqrt{\frac{\lambda_1}{\lambda}} -1 )^m (\frac{\lambda_1}{\lambda} -1)}{2 (\sqrt{\frac{\lambda_1}{\lambda}} + 1)^m },
\end{equation*}
As it can be easily verified, the right side of this equation is a
strictly decreasing function when $\lambda \in [\lambda_0, \lambda_1]$.

As $ | 1- \lambda_0 \cq_m(\lambda_0; \lambda)  | = {\lambda_0} | 1/\lambda_0 - \cq_m(\lambda_0; \lambda)  | $, we will focus on $\mathcal{E}_m ( x; \lambda ) = | 1/x - \cq_m(x; \lambda) | $.  According to the three term recurrence of $\cq_m(x; \lambda)$~\cite[
Equation~(2.13)]{2012KrausJ_VassilevskiP_ZikatanovL-aa}, we have
\begin{equation*}
\cq_{m+1}(x; \lambda) - \cq_{m}(x; \lambda) = \delta^2 (\cq_m(x; \lambda) - \cq_{m-1}(x; \lambda)  )  + c (1 - x \cq_{m}(x; \lambda)),
\end{equation*}
where $\delta = \frac{\sqrt{\kappa} - 1}{\sqrt{\kappa} + 1}$ with $\kappa = \frac{\lambda_1}{\lambda}$ and $c = \frac{4 \sqrt{\mu_0} \sqrt{\mu_1}}{(\sqrt{\mu_0} + \sqrt{\mu_1})^2}$ with $\mu_0 = \frac{1}{\lambda_1}$ and $\mu_1 = \frac{1}{\lambda}$.  Then we have the following three term recurrence of $\mathcal{E}_m(x; \lambda)$
\begin{equation*}
\mathcal{E}_{m+1}(x;\lambda) = 2 y \delta \mathcal{E}_{m}(x; \lambda) - \delta^2 \mathcal{E}_{m-1}(x; \lambda),
\end{equation*}
where $y = \frac{1 + \delta^2 - cx}{2\delta}$.  Because $\mathcal{E}_m$ satisfies the homogeneous three term recurrence relationship, we have the following explicit formula
\begin{equation*}
\mathcal{E}_m(x; \lambda) = - \delta^m \mathcal{E}_0(x;\lambda) U_{m-2}(y) + \delta^{m-1} \mathcal{E}_1(x; \lambda) U_{m-1}(y),
\end{equation*}
where
\begin{equation*}
\mathcal{E}_0(x; \lambda) = \frac{1}{x} - \frac{1}{2}(\mu_0 + \mu_1); \quad \mathcal{E}_{1}(x; \lambda) = \frac{1}{x} - \left( \frac{1}{2} (\sqrt{\mu_0} + \sqrt{\mu_1})^2 - \mu_0 \mu_1 x \right).
\end{equation*}
Using the software package \texttt{Mathematica} (see the output below), we can
verify that $\mathcal{E}_m(\lambda_0; \lambda) > 0$ and is a strictly
increasing function when $\lambda \in [\lambda_0, \lambda_1]$, and
therefore, so is
\begin{equation*}
| 1 - \lambda_0 \cq_m(\lambda_0; \lambda) | = \lambda_0 \mathcal{E}_m(\lambda_0; \lambda).
\end{equation*}
Therefore, because $ | 1- \lambda_1 \cq_m(\lambda_1; \lambda)| $ is
strictly decreasing and $ | 1 - \lambda_0 \cq_m(\lambda_0; \lambda) |
$ is strictly increasing when $\lambda \in [\lambda_0, \lambda_1]$, $
\min_{\lambda \in [\lambda_0, \lambda_1]}\max \{ |1 - \lambda_1 \cq
_m(\lambda_1; \lambda)|, | 1- \lambda_0 \cq_m(\lambda_0; \lambda) |
\}$ has a unique minimizer $\lambdastar$
satisfying~\eqref{eqn:opt-lambda}.  Moreover, according to
Lemma~\ref{lem:max-end}, this minimizer $\lambdastar$ is the unique
minimizer of the minimization problem~\eqref{eqn:min-optimal-lambda}.

Thus, Result~\ref{thm:best-lower} and~\eqref{eqn:opt-lambda}, show
that to find $\lambdastar$ satisfying~\eqref{eqn:opt-lambda}
we need to solve a one-dimensional nonlinear
equation. This value of $\lambdastar$ results in the best convergence
rate for the polynomial smoother for the given range of frequencies.


At the end of this subsection, we provide the \texttt{Mathematica}
code which is used to verify Result \ref{thm:best-lower}.  Without
loss of generality, we assume that $\lambda_1 = 1$, introduce $t =
\sqrt{\kappa}$, and change the formulas of $y$, $\mathcal{E}_0$, and
$\mathcal{E}_1$ correspondingly.  In the \texttt{Mathematica} code, we
first input the variables and then compute the derivative of
$\mathcal{E}_m$ with respect to $t$.

\begin{shaded}\begin{code}
In[1] := delta = (t-1)/(t+1);
	 y = (1+t^2 - 2*t*x)/(t^2 -1);
	 E0 = -0.5*t^2 + ((1/x) - 0.5);
	 E1 = -0.5*(t+1)^2 + t^2*x + 1/x;
	 Em = delta^(m-1)*E1*ChebyshevU[m-1,y]
	       - delta^m*E0*ChebyshevU[m-2,y];
	 Minimize[{Em, 1$\leq$t$\leq$1/Sqrt[x], 0$<$x$\leq$1, 3$\leq$m}, {t, x, m}];
Out[2] = $\left \{3.51279 \times 10^{-22}, \ \{ \text{t} \rightarrow 1.03272, \ \text{x} \rightarrow 0.917826, \ \text{m} \rightarrow 18.11 \} \right\} $
In[3] := Maximize[{$\partial_x$(Em), 1$\leq$t$\leq$1/Sqrt[x], 0$<$x$\leq$1, 3$\leq$m}, {t, x, m}];
Out[4] = $\left \{-3.77826 \times 10^{-48}, \ \{ \text{t} \rightarrow 1.0027, \ \text{x} \rightarrow 0.919534, \ \text{m} \rightarrow 43.6199 \} \right\} $
\end{code}
\end{shaded}

From the output displayed above, we can see that $\mathcal{E}_m > 0$
and the derivative of $\mathcal{E}_m$ with respect to $t$ is negative
and, therefore, $\mathcal{E}_m$ is strictly decreasing, which means
$\mathcal{E}_m$ is strictly increasing with respect to $\lambda$.
This concludes the heuristic justification of Result~\ref{thm:best-lower}.

\section{Numerical Experiments\label{sect:numerics}}

In this section we present numerical results on smoothing properties
of the polynomial smoothers and results on two- and multi-level
methods which use these smoothers in combination with aggressive
coarsening strategy.  When the polynomial smoother is the best
polynomial approximation of $x^{-1}$ or the Chebyshev polynomial, the
tests also include study of the behavior of the smoothers and the
multigrid algorithms with respect to the parameters involved in
defining these polynomial smoothers (e.g. $\lambda_0$ as discussed
earlier in Section~\ref{sect:bounds_LFA}).

To fix ideas we consider as our model problem the Poisson equation on
a domain $\Omega$ with homogeneous Dirichlet boundary
conditions
\begin{eqnarray}\label{model_problem}
{\mathcal A}u({\mathbf x}) = -\Delta u({\mathbf x}) &=& f({\mathbf x}),\quad {\mathbf x}\in \Omega \subseteq {\mathbb R}^d, \; d = 2,3,\\
u({\mathbf x}) &=& 0, \quad {\mathbf x} \in \partial \Omega.\label{model_problem_bc}
\end{eqnarray}
The domains $\Omega$ which we consider are the unit cubes in
$\Reals{d}$, $d=2,3$ or a triangular domain in $\Reals{2}$.  We group
the numerical tests in this section as follows. In
Section~\ref{sec:smooth} we present results on the smoothing property
of the polynomial smoothers. The two-grid and multigrid results are
shown in Section~\ref{sec:cycles} . As discretization technique on
rectangular grids we use the standard $5$-point finite difference
discretization of
problem~(\ref{model_problem})-(\ref{model_problem_bc}). At the end we
also show that our results can be easily generalized to continuous
linear finite element discretizations on triangular grids (see
Section~\ref{sect:triangular_numerics}, and
Section~\ref{sect:triangular_grids}).  In the computations, we use
bilinear interpolation for rectangular grids and linear interpolation
on triangular grids.

One notation that is needed is on the corresponding cycling
strategy. We call $k$-coarsening a coarsening procedure which results
in a coarser grid size of $2^kh$ if the size of the fine grid is $h$.
For multigrid method with total number of levels $L$, this means that
the grid size on level $\ell$ is $2^{k\ell}h_L$ if $L$-th level is the
finest grid level.  Such coarsening is called aggressive since a
coarser grid  has $2^{-kd}n$ degrees of freedom if the
spatial dimension is $d$ and the fine grid has $n$ degrees of freedom.

\subsection{Comparison of smoothing rates}\label{sec:smooth}
The crucial role of the smoothing as one main component in a multigrid
process is well known (see the classical paper by
A.~Brandt~\cite{1977BrandtA-aa}). 
In fact, as shown in this pioneering work on multigrid methods, the
smoothing analysis can provide a first, and often accurate, estimate
on the convergence factor of the overall method.

In this section, we compare the smoothing properties of the polynomial
smoothers introduced in Section~\ref{sect:poly-smoothers}: (a) the
scaled classical Chebyshev polynomial; (b) the smoothed aggregation
polynomial, denoted here with (SA); and (c) the best polynomial
approximation to $x^{-1}$, denoted here with (\zipo).  For the latter
polynomial we provide two different values for the lower bound of the
eigenvalue interval, $\lambda_0$: (1) the choice of $\lambda_0$ given
by the LFA (see~\eqref{lambda_0_1}); and (2) the optimal choice of
$\lambda_0$, given in Section~\ref{sect:choice_lambda}, and denoted by
$\lambdastar$.  The results for the Chebyshev polynomial smoothers
  are for the interval $\lbrack \lambda_0,\lambda_1\rbrack$ with the value of
  $\lambda_0$ determined from~\eqref{lambda_0_1}. In Table~\ref{tab:poly_smooth_compare2d}, we present the
smoothing factors predicted by the local Fourier analysis for these
polynomial smoothers in $2D$. For the LFA estimates we specify as
``low frequencies'' the eigenmodes corresponding to a lattice of size
$2^k h$, where $h$ is the fine grid mesh size. In this setting, larger
values of $k$ correspond to more aggressive coarsening.  For the case
of \zipo smoother this in turn requires higher degree of the
polynomial.
\begin{table}[htdp]
\caption{Smoothing rate comparison of different polynomial smoothers ($\lambda_1 = 2$) in 2D}\label{tab:poly_smooth_compare2d}
\begin{center}
\begin{tabular}{|l|c|c|c|c|c|c|}
\hline \hline
 & Chebyshev & SA & \zipo & \zipo ($\lambdastar$) &  $\lambda_0$  &$\lambdastar$ \\ \hline
$k=1$, Degree $= 2$ &  0.074	       & 0.233  &	0.167  & 0.100 &  0.500 & 0.598  \\
$k=2$, Degree $= 6$ &  0.041	       & 0.221  &	0.226  & 0.086 &  0.146 &0.202  \\
$k=3$, Degree $= 17$ &  0.014	       & 0.172  &	0.230  & 0.053 &  0.038 &0.057  \\
\hline \hline
\end{tabular}
\end{center}
\end{table}%
We use the same degree for the Chebyshev and also SA polynomials and
we observe that the scaled Chebyshev polynomial provides the best
smoothing factors for any of the choices of aggressive
coarsening. Also we observe that choosing the optimal value of
$\lambdastar$ for the smoothing of the $1/x$ polynomial improves
significantly the smoothing rate for this smoother.

Next, similar results are presented in
Table~\ref{tab:poly_smooth_compare3d} for the three-dimensional
case. Again, we can draw the same conclusions since Chebyshev
polynomial results in the best smoothing rates (even better than in
the two-dimensional case) and the optimal choice of
$\lambda_0=\lambdastar$ appears to be crucial to obtain good convergence
factors for the \zipo smoother.
\begin{table}[htdp]
  \caption{Smoothing rate comparison of different polynomial smoothers ($\lambda_1
    = 2$) in 3D}\label{tab:poly_smooth_compare3d}
\begin{center}
\begin{tabular}{|l|c|c|c|c|c|c|}
\hline \hline
	    			 & Chebyshev & SA & \zipo & \zipo ($\lambdastar$) &  $\lambda_0$  &$\lambdastar$ \\ \hline
$k=1$, Degree $= 3$ &  0.062  & 0.227  &  0.185 & 0.097 &  0.333& 0.419 \\
$k=2$, Degree $= 9$ &  0.022  & 0.215  &  0.171 & 0.059 & 0.976 &  0.134\\
$k=3$, Degree $= 22$ & 0.011 &  0.148 & 0.268 & 0.051 & 0.025 & 0.039 \\
\hline \hline
\end{tabular}
\end{center}
\end{table}%

\subsection{Multigrid cycles}\label{sec:cycles}
While the smoothing analysis results presented in previous section
give us some insight about the performance of the method, to properly
study the behavior of multigrid algorithm we need to involve the
coarse grid correction into the tests.  We first present a two-grid
analysis (LFA) which takes into account the effect of transfer
operators and the rest of the components of the coarse-grid correction
part of the algorithm. We report the two-grid convergence rates
obtained by LFA as well as the performance of a W-cycle solver using
the polynomial smoothers. We present the results for two-grid method
and W-cycle for $1$-coarsening, i.e.  all levels are involved in the
coarse grid correction and there is no skipping of levels.

Next, we show results on the behavior of multigrid V-cycles. The tests are
carried out for $k$-coarsening strategies, with $k=1,2,3$, namely we
study the performance of the algorithms with respect to coarsening
strategy with different ``aggressiveness''.

For the two- and multi-grid tests we only compare Chebyshev polynomial
and $1/x$ polynomial since from the previous smoothing rate tests, we
can assert that these two polynomials outperform the SA polynomials as
smoothers. Also all the convergence factors presented in the
tables are computed with only one smoothing step.

In the tables below, we denote by $\rho_{LFA}$ the two-grid
convergence factors predicted by LFA.  The results for the
$k$-coarsening strategies are presented in
Table~\ref{tab_two_grid_2D}. In order to validate these results, we
also show in this table the corresponding experimentally computed
asymptotic two-grid convergence factors $\rho_W$. The lower and upper
bounds of the eigenvalue interval $[\lambda_0,\lambda_1]$ provided by
the LFA are also given, together with the optimal value $\lambdastar$,
computed as in~(\ref{eqn:opt-lambda}), and with the minimal degree of
the polynomial smoothers.  We observe that according to
Table~\ref{tab_two_grid_2D} the local Fourier analysis provides sharp
estimates for the convergence rates in all the cases.

Finally, we remark on the choice of parameter $\lambdastar$. As seen
this choice improves the convergence
factors provided by \zipo smoother. The improvement is more
significant for $k$-coarsening with larger $k$. This makes the
convergence rates of \zipo very close to those given by
the best option, namely, the scaled and shifted classical Chebyshev polynomial.
\begin{table}[htdp]
\caption{Two-grid convergence rate for different polynomial smoothers in 2D}
\begin{center}
\begin{tabular}{|c|c|c|c|c|c|c|c|c|c|c|}
\hline
\hline
\multicolumn{5}{|c|}{} & \multicolumn{2}{c|}{Chebyshev} & \multicolumn{2}{c|}{\zipo} & \multicolumn{2}{c|}{\zipo($\lambdastar$)} \\ \hline
&  $\lambda_0$ & $\lambdastar$ & $\lambda_1$ & Degree & $\rho_{LFA}$ & $\rho_{W}$ & $\rho_{LFA}$ & $\rho_{W}$ & $\rho_{LFA}$ & $\rho_{W}$\\  \hline
$k=1$ &	$0.5$	& $0.598$ & $2.0$ & $2$  & $0.125$ & $0.126$ & $0.166$ & $0.166$ & $0.134$ & $0.134$ \\	
$k=2$ &	$0.146$	& $0.202$ & $2.0$ & $6$  & $0.156$ & $0.155$ & $0.221$ & $0.225$ & $0.166$ & $0.165$ \\	
$k=3$ &	$0.038$ & $0.057$ & $2.0$ & $17$ & $0.137$ & $0.137$ & $0.227$ & $0.227$ & $0.148$ & $0.149$ \\	
\hline \hline
\end{tabular}
\end{center}
\label{tab_two_grid_2D}
\end{table}%
It is also interesting to choose a value of $\lambda_0$ (left end of
the high frequency interval) which is optimal with respect to the
overall two-grid convergence of the method. Note that, this is
a different procedure than what we have done before: choosing $\lambdastar$ such
that the smoothing factor is the best possible. This value has been
computed by using the local Fourier analysis, for both Chebyshev and
\zipo polynomials. More precisely, we find the optimal $\lambdastar$
  not only by monitoring the convergence of the smoother, but also the
convergence of the two-grid method. In Table~\ref{tab_optimal_TL} we present
two-grid convergence factors for this case as predicted by the LFA and we
also show the experimentally computed W-cycle rates.  We observe
results similar to the ones already discussed in Table~\ref{tab_two_grid_2D}.
\begin{table}[htdp]
  \caption{Two-grid convergence rate for Chebyshev and \zipo
    polynomial smoothers.  The
    value of $\lambda_0$ is chosen by applying LFA for the overall
    convergence factor of the two-grid method.
}
\begin{center}
\begin{tabular}{|c|c|c|c|c|c|c|c|c|}
\hline
\hline
\multicolumn{3}{|c|}{} & \multicolumn{3}{c|}{Chebyshev($\lambda_0^{Optimal-TL}$)} & \multicolumn{3}{c|}{\zipo($\lambda_0^{Optimal-TG}$)} \\ \hline
& $\lambda_1$ & Degree &  $\lambda_0$ & $\rho_{LFA}$ & $\rho_{W}$ & $\lambda_0$ & $\rho_{LFA}$ & $\rho_{W}$ \\  \hline
$k=1$ &	$2.0$ & $2$	 & $0.405$ & $0.111$ & $0.113$ & $0.550$ & $0.128$ & $0.128$ \\
$k=2$ &	$2.0$ & $6$	 & $0.095$ & $0.138$ & $0.140$ & $0.167$ & $0.152$ & $0.153$ \\
$k=3$ &	$2.0$ & $17$ & $0.019$ & $0.100$ & $0.100$ & $0.045$ & $0.133$ & $0.133$ \\
\hline \hline
\end{tabular}
\end{center}
\label{tab_optimal_TL}
\end{table}%

Further, we test the behavior or V-cycle for the different polynomial
smoothers and different coarsening strategies with varying
  coarsening ratios. We consider a V-cycle with one pre- and one
post-smoothing step. The corresponding results are shown in
Table~\ref{V_cycle_2D}, and we observe that for both polynomials the
obtained convergence rates are $\approx 0.1$, that is, we have fast
convergence reducing the error in energy norm by an order of magnitude
per iteration.
\begin{table}[htdp]
  \caption{Multigrid convergence rates for a V(1,1)-cycle in
    2D. The
    value of $\lambda_0$ is chosen by LFA
    (equation~\eqref{lambda_0_1}).
}
\begin{center}
\begin{tabular}{|c|c|c|c|c|c|}
\hline \hline
	  & $\lambda_0$ & $\lambdastar$ & \zipo ($\lambda_0$) & \zipo ($\lambdastar$) & Chebyshev \\  \hline
$k=1$, Degree$=2$&    0.5        &   0.598   &  0.103     &  0.114  &        0.111 \\
$k=2$, Degree$=6$&    0.146   &   0.202   &   0.088    &   0.103   &        0.098		 \\
$k=3$, Degree$=17$&    0.038 &   0.057   &   0.069    &   0.083   &        0.076	 \\
\hline \hline
\end{tabular}
\end{center}
\label{V_cycle_2D}
\end{table}%

In regard to the three-dimensional case the tests
provide a solid basis for conclusions analogous to the ones obtained
for the two-dimensional case.
The results for V(1,1)-cycle in three dimensions are shown in
Table~\ref{V_cycle_3D}.
Also, as in the 2D case, we include results on the behavior with
respect to the choice of the parameter $\lambda_0$, computed by
LFA, and also its optimal version, $\lambdastar$ given
by~(\ref{eqn:opt-lambda}). As in the 2D case we
observe excellent convergence rates.

\begin{remark}
  We remark that in Table~\ref{V_cycle_2D} and
  Table~\ref{V_cycle_3D} the value of $\lambdastar$ is calculated so
  that it optimizes the smoothing factor only. However, since the overall
  algorithm involves in addition a coarse grid correction, we cannot expect
  that $\lambdastar$ will yield an optimal convergence rate for
  the resulting multilevel method (since it does not take the coarse grid
  correction into account).  In short, the proposed choice for the value of $\lambdastar$
  guarantees a better smoothing rate, but this does not necessarily lead to a better V-cycle
  convergence rate. This is shown in Table~\ref{V_cycle_2D} and
  Table~\ref{V_cycle_3D}. We note that the effects of this observed phenomenon are minimal in practice and in general the convergence
  rates of the methods for $\lambda_0$ and for $\lambdastar$ are generally very close.
\end{remark}

\begin{table}[htdp]
\caption{Multigrid convergence rates for a V(1,1)-cycle in 3D.
The
    value of $\lambda_0$ is chosen by LFA
    (equation~\eqref{lambda_0_1}).
}
\begin{center}
\begin{tabular}{|c|c|c|c|c|c|}
\hline \hline
& $\lambda_0$ & $\lambdastar$ & \zipo ($\lambda_0$) & \zipo ($\lambdastar$) & Chebyshev \\  \hline
$k=1$, Degree$=3$&    0.333  &   0.419   &  0.101     &  0.115  &   0.110	 \\
$k=2$, Degree$=9$&    0.098  &   0.134  &   0.084    &   0.099   &   0.094	 \\
$k=3$, Degree$=22$&  0.025  &   0.039   &   0.071    &  0.090  &   0.079	 \\
\hline \hline
\end{tabular}
\end{center}
\label{V_cycle_3D}
\end{table}%

\subsection{Extension to triangular grids}\label{sect:triangular_numerics}
All the ideas and techniques we introduced earlier can be extended to the
case of triangular grids. This allows us to study how the geometry
of the grid influences the convergence rates and what are the optimal
parameters for the polynomial smoothers with respect to the geometry
of the grid.

We consider continuous linear finite element discretization of
the problem~(\ref{model_problem})-(\ref{model_problem_bc}) on a
structured triangular grid characterized by two angles $\alpha$ and
$\beta$, as explained in Section~\ref{sect:triangular_grids}. As mentioned in
that section, LFA can be applied on triangular
grids. This gives us a tool to automatically choose
polynomial degrees and suitable values of $\lambda_0$, as we did for
rectangular grids. As intergrid transfer
operator we use the natural inclusion (linear interpolation).

First example is on a grid with equilateral triangles (the
domain $\Omega$ is also such a triangle). In
Table~\ref{lfa_linear_equilateral} we display the two-grid convergence
factors provided by the LFA for the multigrid
algorithms based on the polynomial smoothers we consider. The
results are shown for $k$-coarsening for different values of $k$.
We display also the values of $\lambda_0$
and $\lambdastar$ in Table~\ref{lfa_linear_equilateral}. Similarly to
the observation made earlier, the Chebyshev polynomial smoothers
seem to provide the best convergence factors, followed by the \zipo smoother with the
optimal choice of $\lambdastar$. Here the optimal choice of $\lambdastar$ again improves significantly the
smoothing properties of the \zipo smoother. Also, we observe
that the minimal degrees for the polynomial smoothers are very close
to those used in the rectangular case, and even slightly lower.
\begin{table}[htdp]
\caption{Two-grid convergence factors on an equilateral triangle by
  using linear interpolation. The value of $\lambda_1$ is $\frac32$
 and the
    value of $\lambda_0$ is chosen by LFA (equation~\eqref{lambda_0_1}).
}
\begin{center}
\begin{tabular}{|c|c|c|c|c|c|}
  \hline \hline
  & $\lambda_0$ & $\lambdastar$ & \zipo ($\lambda_0$) & \zipo
  ($\lambdastar$) & Chebyshev
  \\  \hline
  $k=1$, Degree$=1$ &  0.529  &   0.623   &  0.212     &  0.138  &   0.129	 \\
  $k=2$, Degree$=5$ &  0.148  &   0.195   &  0.175     &  0.101  &   0.102	 \\
  $k=3$, Degree$=14$&  0.038  &   0.056   &  0.236     &  0.091  &   0.086	 \\
  \hline \hline
\end{tabular}
\end{center}
\label{lfa_linear_equilateral}
\end{table}

Next we fix an structured isosceles triangular grid with base angle
$\frac{4}{9}\pi$. This triangulation has a relative small third angle
equal to $\frac{\pi}{9}$ and this induces some grid anisotropy.
We have applied the local Fourier
analysis to the resulting discretization in such a grid and
the two-grid convergence factors are displayed in
Table~\ref{lfa_linear_isosceles} for the different $k$-coarsening
strategies and different smoothers. We observe that we are able to
obtain good convergence factors at the price of increasing the
degree of the polynomial. This shows that
much ``stronger'' smoothers are needed for anisotropic problems, a
fact that is known in the multigrid community. In fact,  when standard coarsening is
considered, coupled smoothers (line-wise relaxation) are preferred
against the standard point-wise relaxations as Jacobi,
Gauss-Seidel. We show here that with appropriate polynomial degree,
polynomial smoothers also result in good convergence rates and in
addition they have advantage in parallel computation (compared to
block Gauss-Seidel or block Jacobi method).
\begin{table}[htdp]
\caption{Two-grid convergence factors on an isosceles triangle with
  common angle $\frac{4}{9}\pi$, by using linear interpolation. The
  value of $\lambda_1$ is $\frac{17}{9}$ and the
    value of $\lambda_0$ is chosen by LFA
    (equation~\eqref{lambda_0_1}).
}
\begin{center}
\begin{tabular}{|c|c|c|c|c|c|}
\hline \hline
& $\lambda_0$ & $\lambdastar$ & \zipo ($\lambda_0$) & \zipo ($\lambdastar$) & Chebyshev \\  \hline
$k=1$, Degree$=8$ &  0.112  &   0.151   &  0.151     &  0.079  &   0.064	 \\
$k=2$, Degree$=18$&  0.033  &   0.049   &  0.261     &  0.101  &   0.092	 \\
$k=3$, Degree$=43$&  0.009  &   0.014   &  0.616     &  0.095  &   0.086	 \\
\hline \hline
\end{tabular}
\end{center}
\label{lfa_linear_isosceles}
\end{table}

\section{Conclusions\label{sect:conclusions}} We have devised a simple
technique based on the local Fourier analysis which allows us to
construct polynomial smoothers with optimal smoothing factors and
parameters tied also to the damping of the error on the coarser grids.
The theoretical and the numerical results clearly confirm that LFA is
a useful tool in designing efficient multigrid algorithms with
more aggressive coarsening and polynomial smoothing.  Also as it is
seen in the numerical examples section the degree of the polynomials
grows linearly with respect to the coarsening ratio. This is not
surprising since the more we coarsen, the more smoothing is needed to
cover the whole high-frequency interval. This technique can be applied
on structured as well as semi-structured grids. These methods are
suitable for parallelization since they only involve matrix vector
multiplications and the local Fourier analysis automates the parameter
choice. Studies and comparisons with multicolored Gauss-Seidel and SOR
smoothers are part of the planned research.

\section*{Acknowledgments}
The work of James Brannick was supported in part by NSF grants
DMS-1217142 and DMS-1320608 and by Lawrence Livermore National
Laboratory through subcontract B605152.  The work of Carmen Rodrigo is supported in part by
the Spanish project FEDER/MCYT MTM2010-16917 and the DGA (Grupo
consolidado PDIE).  The research of Ludmil Zikatanov is supported in
part by NSF grant DMS-1217142, and Lawrence Livermore National
Laboratory through subcontract B603526.
Carmen Rodrigo gratefully acknowledges the hospitality of the Center for Computational Mathematics and Applications and the Department of Mathematics of The Pennsylvania State University, where this research was partly carried out.

\bibliographystyle{plain}
\bibliography{bibzipo}

\end{document}